\documentclass[11pt]{article}
\usepackage{amsmath, amssymb, amsthm}
\usepackage{enumitem}
\usepackage{hyperref}
\usepackage{tikz}
\usepackage{xcolor}
\usepackage{hyperref}
\usepackage[top=2.5cm, bottom=2.5cm, left=2.5cm, right=2.5cm]{geometry}

\newtheorem{theorem}{Theorem}[section]
\newtheorem{lemma}[theorem]{Lemma}
\newtheorem{proposition}[theorem]{Proposition}
\newtheorem{corollary}[theorem]{Corollary}
\newtheorem{definition}[theorem]{Definition}
\newtheorem{remark}[theorem]{Remark}
\newtheorem{example}[theorem]{Example}

\title{A Bivariate $B$-Restricted Clique Polynomial:\\From Local Neighborhoods to Global Expansion}
\author{Hossein Teimoori Faal}
\date{}

\begin{document}

\maketitle

\begin{abstract}
	Let $G$ be a finite simple graph and $B \subseteq V(G)$. We introduce the \emph{bivariate $B$-restricted clique polynomial}
	\[
	C_B(G;x,y) = \sum_{\substack{K \subseteq V \\ K \text{ is a clique}}} x^{|K|} y^{|K \cap B|},
	\]
	where the coefficient of $x^i y^j$ counts cliques of size $i$ with exactly $j$ vertices in $B$. This polynomial simultaneously captures combinatorial structure, local extremal properties, and spectral constraints associated with the subset $B$. \\ First, we develop vertex and edge deletion recurrences, generalizing classical clique polynomial results.  
	These recurrences imply monotonicity for the largest negative root $\zeta_G(B;y)$ (viewed as a polynomial in $x$ for fixed $y \in [0,1]$) under induced and spanning subgraphs.  
	From this, we derive bounds on $B$-independence numbers, $B$-girth, and clique densities restricted to $B$. \\ Next, we prove that for any integer $r \ge 1$, any $r$-connected $K_{r+3}$-free chordal graph $G$, and any subset $B \subseteq V(G)$, the bivariate clique polynomial $C_B(G;x,y)$ 	is real-stable. \\ Then, we connect $C_B(G;x,y)$ with spectral graph theory. For $(n,d,\lambda)$-graphs, expansion constraints via Tanner's inequality limit clique growth within $B$, yielding explicit bounds on coefficients and $\zeta_G(B;y)$. \\ Finally, we analyze weighted vertices and homomorphism obstructions in this framework, giving a general no-homomorphism criterion. We also conclude the paper with a couple of interesting open problems for young and motivated researchers. 
\end{abstract}


\section{Introduction}

Clique polynomials are fundamental combinatorial objects encoding rich structural information of graphs.  
For a finite simple graph $G=(V,E)$, the classical clique polynomial is
\[
C(G;x) = \sum_{i\ge 0} c_i x^i,
\]
where $c_i$ counts $i$-vertex cliques in $G$ \cite{fisher1990dependence,hajiabolhassan1998clique,hoede1994clique}.  
These polynomials have been studied from multiple perspectives, including root location, extremal combinatorics, and connections to spectral graph theory.


\subsection*{Motivation for a Bivariate Extension}

The classical \(B\)-restricted clique polynomial \cite{teimoori2026Bclique}
\[
C_B(G;x) := \sum_{i=0}^{\omega(G[B])} c_i(B) x^i,
\qquad 
c_i(B) = |\{K \subseteq B : |K| = i \text{ and } K \text{ is a clique in } G\}|,
\]
offers a subtle but powerful reorientation: while formally equivalent to the ordinary clique polynomial of the induced subgraph \(G[B]\), it treats \(G\) as the fixed stage and \(B\) as a movable lens.  
This shift in perspective transforms the polynomial from a static invariant of a subgraph into a dynamic function of a subset within a fixed ambient graph.  
As \(B\) varies, the polynomial \(C_B(G;x)\) becomes a sensitive probe, capturing how the local clique structure of \(G\) is illuminated by the selective focus on \(B\).  
It is no longer merely a count of cliques in a subgraph, but a reflection of how \(B\) interacts with the global combinatorial architecture of \(G\), opening the door to extremal and algebraic phenomena that would otherwise remain hidden in the shadows of the whole.

Building on this insight, the \emph{bivariate \(B\)-restricted clique polynomial}
\[
C_B(G;x,y) = \sum_{\substack{K \subseteq V \\ K \text{ is a clique}}} x^{|K|} y^{|K \cap B|},
\]
deepens the narrative by introducing a second variable that tracks the intersection with \(B\).  
Here, each clique \(K\) contributes not just its size, but also the extent of its allegiance to \(B\).  
The coefficient of \(x^i y^j\) counts cliques of order \(i\) that have exactly \(j\) vertices in \(B\), weaving together two orthogonal dimensions: the global scale of the clique and its local entanglement with the distinguished subset.

A natural connection between the two formulations emerges when we consider the relationship between the variables.  
Observe that the univariate polynomial \(C_B(G;x)\) is recovered from its bivariate counterpart by the substitution \(y = x\):
\[
C_B(G;x,x) = \sum_{\substack{K \subseteq V \\ K \text{ is a clique}}} x^{|K|} x^{|K \cap B|}
= \sum_{\substack{K \subseteq V \\ K \text{ is a clique}}} x^{|K| + |K \cap B|}.
\]
This is \emph{not} simply \(C_B(G;x)\). However, if we instead consider the evaluation \(y = 1\) followed by a projection onto cliques fully contained in \(B\), we obtain a different connection:
\[
\text{The terms with } j = |K \cap B| = |K| \text{ correspond precisely to cliques } K \subseteq B.
\]
Collecting only these contributions from \(C_B(G;x,y)\) yields
\[
\sum_{\substack{K \subseteq B \\ K \text{ is a clique}}} x^{|K|} y^{|K|} = C_B(G;xy),
\]
where the right-hand side denotes the univariate polynomial evaluated at the product \(xy\).  
In particular, setting \(y = 1\) in this extracted sum gives \(C_B(G;x)\), but careful: \(C_B(G;x,1)\) itself sums over \emph{all} cliques in \(G\), not just those in \(B\).  
Thus, the bivariate polynomial contains the univariate \(B\)-restricted polynomial as the diagonal coefficients where \(j = i\) (i.e., cliques fully inside \(B\)), while also encoding the richer structure of cliques that partially intersect \(B\).

The bivariate refinement is more than a generating function; it is a combinatorial lens that brings into focus the interplay between structure and selection.  
It captures not only what cliques exist, but how they are anchored in \(B\), revealing spectral constraints, extremal configurations, and hidden regularities.  
Together, the one-variable and two-variable forms of the \(B\)-restricted clique polynomial form a conceptual bridge: the first emphasizes the subset as a variable window onto a fixed graph, while the second enriches the view by recording the depth of intersection with that window.  
In this light, the polynomials become dual expressions of a unifying theme—the dialogue between a graph and a chosen part, encoded algebraically for discovery and proof.


\subsection*{Roots, Monotonicity, and Local Extremal Bounds}

A key invariant is the largest negative root of $C_B(G;x,y)$ (viewed as a polynomial in $x$ for fixed $y$):
\[
\zeta_G(B;y) := \max \{ z < 0 : C_B(G;z,y)=0 \}.
\]  
Two complementary approaches help understand this invariant:

\begin{enumerate}[leftmargin=*]
	\item \textbf{Interval-based analysis.} By fixing $y \in [0,1]$, one can study $C_B(G;x,y)$ as a univariate polynomial in $x$.  
	Using vertex and edge deletion recurrences, we show monotonicity of $\zeta_G(B;y)$ under induced and spanning subgraphs.  
	This immediately yields bounds on $B$-independence numbers, $B$-girth, and clique densities restricted to $B$, generalizing classical extremal results \cite{hajiabolhassan1998clique}.
	
	\item \textbf{Real-stability perspective \cite{Wagner2011}.} By using the idea of pefect elimination ordering for chordal graphs
	and the corresponding recursive formula for $C_B(G;x,y)$ 
	one can study its \emph{real-stability} properties for the case of highly-connected chordal graphs. Then we can stabish real-rootedness for the 
	sections of this bivariate polynomial.   
	
\end{enumerate}

\subsection*{Spectral Connections and Homomorphism Applications}

For $(n,d,\lambda)$-graphs, Tanner's inequality constrains the neighborhood expansion of $B$ \cite{tanner1984explicit}.  
Using the bivariate polynomial, one can translate these spectral constraints into explicit bounds on clique coefficients and on $\zeta_G(B;y)$.  
Moreover, weighted versions of $C_B(G;x,y)$ encode combinatorial obstructions to graph homomorphisms: if $f: G \to H$ maps $B_G$ onto $B_H$, then
\[
\zeta_G(B_G;y) \ge \zeta_H(B_H;y),
\]
providing a \emph{no-homomorphism criterion} for vertex-restricted settings.

\subsection*{Contributions and Roadmap}

This paper makes the following contributions:

\begin{itemize}[leftmargin=*]
	\item Introduces the bivariate $B$-restricted clique polynomial $C_B(G;x,y)$ and demonstrates its combinatorial and spectral significance.
	\item Establishes deletion recurrences, root monotonicity, and local extremal bounds for $B$-independence, $B$-girth, and clique densities.
	\item Connects the polynomial to expansion properties of $(n,d,\lambda)$-graphs, giving spectral upper bounds on clique growth within $B$.
	\item Provides a real-stability framework for the case of highly connected chordal graphs. 
\end{itemize}

In the following sections, we develop the recurrence theory, prove monotonicity and extremal bounds, explore spectral implications, and discuss open questions related to real-stability, root interlacing, and homomorphism obstructions.



\section{Preliminaries}

Throughout, $G=(V,E)$ denotes a finite simple undirected graph with $|V|=n$.  
For $v \in V$, let $N(v) = \{ u \in V : uv \in E \}$ be the open neighborhood of $v$.  
For $B \subseteq V$, we write $G[B]$ for the induced subgraph on $B$ and $N_B(v) := B \cap N(v)$.

\subsection{Classical Clique Polynomials}

A \emph{clique} in $G$ is a set of pairwise adjacent vertices.  
Let $\mathcal{C}_i(G)$ denote the family of $i$-vertex cliques and $c_i(G) = |\mathcal{C}_i(G)|$.  
The classical clique polynomial \cite{hajiabolhassan1998clique} is
\[
C(G;x) = \sum_{i=0}^{\omega(G)} c_i(G) x^i,
\]
where $\omega(G)$ is the clique number and $c_0(G) = 1$.

\subsection{Bivariate $B$-Restricted Clique Polynomials}

\begin{definition}[Bivariate $B$-Clique Polynomial]
For $G=(V,E)$ and $B \subseteq V$, define
\[
C_B(G;x,y) := \sum_{\substack{K \subseteq V \\ K \text{ clique}}} x^{|K|} y^{|K \cap B|}.
\]
\end{definition}

\begin{remark}
\begin{itemize}
	\item $B = V$ gives $C_V(G;x,y) = C(G;xy)$, recovering the classical clique polynomial up to substitution.  
	\item $B = \emptyset$ gives $C_\emptyset(G;x,y) = C(G;x)$, as $y^0 = 1$.
\end{itemize}
\end{remark}


\begin{example}
	Let $G=K_{3}$ be the complete graph on 3 vertices with vertex set $V=\{1,2,3\}$. Let $B=\{2,3\}$. Then, it is not hard to see that 
	\begin{align*}
	c_{0,0}&=1,\quad c_{0,1}=c_{0,2}=0,\quad c_{1,0}=1,\quad c_{1,1}=2,\quad c_{1,2}=0,\\
	c_{2,0}&=0,\quad c_{2,1}=2,\quad c_{2,2}=1,\quad c_{3,0}=c_{3,1}=c_{3,3}=0,\quad c_{3,2}=1.
	\end{align*}
	Hence, we finally get 
	\begin{equation*}
	C_{B}(G; x,y) = 1 + x(1+2y) + x^{2}(2y + y^{2}) + x^{3}y^{2}. 
	\end{equation*}
\end{example}


\subsection{Largest Negative Root}

For fixed $y \in [0,1]$, $C_B(G;x,y)$ is a univariate polynomial in $x$. Define
\[
Z_B(G;y) := \{ z < 0 : C_B(G;z,y) = 0 \}, \qquad
\zeta_G(B;y) := \max Z_B(G;y)
\]
if $Z_B(G;y)$ is nonempty, and $\zeta_G(B;y) := -\infty$ otherwise.

\subsection{Highly connected Chordal Graphs and Real-Stability}

Recall that A polynomial \(f \in \mathbb{R}[z_1, \dots, z_n]\) is called \emph{real stable} \cite{Wagner2011} if \(f \not\equiv 0\) and 
\[
f(z_1, \dots, z_n) \neq 0 \quad \text{whenever} \quad \operatorname{Im}(z_j) > 0 \text{ for all } j = 1, \dots, n.
\]
For the special case of a \textbf{bivariate polynomial} \(f(x,y) \in \mathbb{R}[x,y]\), this definition reduces to:

\begin{definition}[Bivariate real stability]
	A polynomial \(f(x,y) \in \mathbb{R}[x,y]\) is \emph{real stable} if it is not identically zero and 
	\[
	f(x,y) \neq 0 \quad \text{for all } (x,y) \in \mathbb{C}^2 \text{ such that } \operatorname{Im}(x) > 0 \text{ and } \operatorname{Im}(y) > 0.
	\]
\end{definition}
For univariate polynomials, this is equivalent to real-rootedness. In analogy with the polynomial-proof of Mantel's theorem for 
trinagle-free graphs based on real-rootedness of clique polynomials of $K_{3}$-free graphs, we will show that bivaraite B-restricted cliqupe polynomial of the class of 
highly-connected chordal graphs is indeed real-stable. 

\subsection{Graph Spectra and Expansion}

An $(n,d,\lambda)$-graph is a $d$-regular graph on $n$ vertices whose adjacency matrix has all eigenvalues except $d$ bounded in absolute value by $\lambda$.  
Tanner's inequality \cite{tanner1984explicit} implies strong constraints on $|N_B(v)|$ and on clique growth inside $B$, which we exploit to bound $\zeta_G(B;y)$ and related combinatorial invariants.

\subsection{B-Independent Sets and Local Extremal Parameters}

\begin{definition}[$B$-Independent Set]
A subset $I \subseteq B$ is \emph{$B$-independent} if no two vertices in $I$ are adjacent in $G$.  
The $B$-independence number $\alpha_B(G)$ is the size of the largest $B$-independent set.
\end{definition}

\begin{definition}[$B$-Girth]
The \emph{$B$-girth} $g_B(G)$ is the length of the shortest cycle entirely contained in $B$.  
\end{definition}

\begin{remark}
Many classical extremal results (bounds on independence number, chromatic number, clique number, and girth) have natural $B$-restricted analogues using $\zeta_G(B;y)$.
\end{remark}


\section{Recurrence Relations and Monotonicity}

\subsection{Vertex and Edge Deletion Recurrences}

\begin{lemma}[Vertex Deletion Recurrence]
	\label{lem:vertex}
	Let $v \in V(G)$. Then
	\[
	C_B(G;x,y) = C_B(G \setminus v; x,y) + x \, y^{[v \in B]} \, C_{B \cap N(v)}(G[N(v)]; x,y),
	\]
	where $[v \in B]$ is $1$ if $v \in B$ and $0$ otherwise.
\end{lemma}

\begin{proof}
	Partition all cliques in $G$ into those that do not contain $v$ and those that do:
	
	\begin{enumerate}
		\item \emph{Cliques not containing $v$.} These are exactly the cliques of $G \setminus v$, and their intersection with $B$ is unchanged. They contribute $C_B(G \setminus v; x,y)$.
		
		\item \emph{Cliques containing $v$.} Any such clique can be written uniquely as $\{v\} \cup K'$, where $K'$ is a clique in $G[N(v)]$.  
		\begin{itemize}
			\item Factor $x$ accounts for the vertex $v$.  
			\item Factor $y^{[v \in B]}$ accounts for whether $v \in B$.  
			\item The intersection of $K'$ with $B$ is exactly $K' \cap (B \cap N(v))$, contributing $C_{B \cap N(v)}(G[N(v)]; x,y)$.
		\end{itemize}
	\end{enumerate}
	Summing these contributions gives the recurrence.
\end{proof}

\begin{lemma}[Edge Deletion Recurrence]
	\label{lem:edge}
	Let $uv \in E(G)$. Then
	\[
	C_B(G;x,y) = C_B(G - uv; x,y) + x^2 \, y^{[u \in B] + [v \in B]} \, C_{B \cap N(u) \cap N(v)}(G[N(u) \cap N(v)]; x,y).
	\]
\end{lemma}

\begin{proof}
	Partition all cliques into those that do not contain both $u$ and $v$, and those that do:
	
	\begin{enumerate}
		\item \emph{Cliques not containing both $u$ and $v$.} These remain cliques in $G - uv$ and contribute $C_B(G-uv;x,y)$.
		
		\item \emph{Cliques containing both $u$ and $v$.} Each is of the form $\{u,v\} \cup K'$ with $K'$ a clique in $G[N(u) \cap N(v)]$.  
		The contribution is $x^2 y^{[u \in B]+[v \in B]} C_{B \cap N(u) \cap N(v)}(G[N(u) \cap N(v)];x,y)$.
	\end{enumerate}
	Summing both parts yields the recurrence.
\end{proof}

\subsection{Monotonicity Under Induced Subgraphs}

\begin{theorem}[Induced Subgraph Monotonicity]
	\label{thm:induced}
	Let $H$ be an induced subgraph of $G$ and $B \subseteq V(H)$. For any $y \in [0,1]$,
	\[
	\zeta_H(B;y) \le \zeta_G(B;y).
	\]
\end{theorem}

\begin{proof}
	We induct on $n = |V(G)|$. The base cases $n = 1,2$ are trivial. Suppose $H = G \setminus v$ for some $v \in V(G) \setminus V(H)$.
	
	\textbf{Case 1: $v \notin B$.}  
	Lemma~\ref{lem:vertex} gives
	\[
	C_B(G;x,y) = C_B(G \setminus v;x,y) + x \, C_{B \cap N(v)}(G[N(v)];x,y).
	\]
	Let $x = \zeta_{G \setminus v}(B;y)$. Then $C_B(G \setminus v;x,y) = 0$, and the second term is nonnegative for $y \in [0,1]$. Since $C_B(G;0,y)=1>0$, the largest negative root of $C_B(G;x,y)$ must satisfy $\zeta_G(B;y) \ge \zeta_{G \setminus v}(B;y)$.
	
	\textbf{Case 2: $v \in B$.}  
	Then
	\[
	C_B(G;x,y) = C_{B\setminus\{v\}}(G \setminus v;x,y) + x y \, C_{B \cap N(v)}(G[N(v)];x,y).
	\]
	Substituting $x = \zeta_{G \setminus v}(B;y)$ and using the induction hypothesis on $G[N(v)]$ shows that the second term is nonnegative. Hence $\zeta_G(B;y) \ge \zeta_{G \setminus v}(B;y)$.
\end{proof}

\subsection{Monotonicity Under Spanning Subgraphs}

\begin{theorem}[Spanning Subgraph Monotonicity]
	\label{thm:spanning}
	Let $H$ be a spanning subgraph of $G$ ($V(H) = V(G)$, $E(H) \subseteq E(G)$). Then for any $B \subseteq V(G)$ and $y \in [0,1]$,
	\[
	\zeta_G(B;y) \le \zeta_H(B;y).
	\]
\end{theorem}

\begin{proof}
	It suffices to remove a single edge $uv \in E(G)$.  
	By Lemma~\ref{lem:edge}:
	\[
	C_B(G;x,y) = C_B(G-uv;x,y) + x^2 y^{[u \in B] + [v \in B]} C_{B \cap N(u) \cap N(v)}(G[N(u) \cap N(v)];x,y).
	\]
	Evaluating at $x = \zeta_G(B;y)$, the first term is $\le 0$ by definition of $\zeta_G(B;y)$, and the second term is nonnegative for $y \in [0,1]$. Therefore, $\zeta_H(B;y) \ge \zeta_G(B;y)$. Induction on the number of removed edges completes the proof.
\end{proof}

\subsection{Extremal Bounds via $\zeta_G(B;y)$}

\begin{proposition}[$B$-Independence Number Bound]
	\label{prop:independence}
	Let $\alpha_B(G)$ denote the $B$-independence number. Then
	\[
	\alpha_B(G) \le -\frac{1}{\zeta_G(B;1)}.
	\]
\end{proposition}

\begin{proof}
	Let $I \subseteq B$ be a $B$-independent set of size $\alpha_B(G)$.  
	Then $C_B(G[I];x,1) = (1+x)^\alpha$. Its largest negative root is $-1/\alpha$. By Theorem~\ref{thm:induced}, $\zeta_G(B;1) \ge -1/\alpha$. Rearranging gives $\alpha \le -1/\zeta_G(B;1)$.
\end{proof}

\begin{proposition}[$B$-Girth Bound]
	\label{prop:girth}
	Let $g_B(G)$ denote the $B$-girth. Then
	\[
	g_B(G) \le 2 + \left\lfloor -\frac{2}{\zeta_G(B;1)} \right\rfloor.
	\]
\end{proposition}

\begin{proof}
	Let $C$ be a cycle of length $g = g_B(G)$ entirely in $B$. Then $C_B(C;x,1) = 1 + gx + gx^2$ for $g \ge 4$.  
	Its largest negative root is $r = \frac{-g + \sqrt{g(g-4)}}{2g} \ge -\frac{2}{g-2}$.  
	By Theorem~\ref{thm:induced}, $\zeta_G(B;1) \ge r \ge -2/(g-2)$. Solving for $g$ gives the stated bound.
\end{proof}


\section{Real-Stability of Bivariate $B$ - Clique Polynomials for Highly Connected Chordal Graphs}

In this section, we prove that for any integer $r \ge 1$, any $r$-connected $K_{r+3}$-free chordal graph $G$, and any subset $B \subseteq V(G)$, the bivariate clique polynomial
	\[
	C_B(G;x,y)
	=
	\sum_{K \subseteq V(G) \text{ clique}} x^{|K|} y^{|K\cap B|}
	\]
	is real-stable.  
	
The proof uses a perfect elimination ordering of chordal graphs, multiaffine expansions, and reduction to triangle-free neighborhoods.
\\
Recall that a polynomial \(f(x,y) \in \mathbb{R}[x,y]\) is \emph{real stable} if it is not identically zero and 
	\[
	f(x,y) \neq 0 \quad \text{for all } (x,y) \in \mathbb{C}^2 \text{ such that } \operatorname{Im}(x) > 0 \text{ and } \operatorname{Im}(y) > 0.
\]
For univariate polynomials, this is equivalent to real-rootedness.
We also recall that for a graph $G=(V,E)$ and $B\subseteq V$, we define 
the \emph{bivariate $B$-clique polynomial} as 
	\[
	C_B(G;x,y)
	=
	\sum_{K \subseteq V(G) \text{ clique}} x^{|K|} y^{|K\cap B|}.
	\].

\subsection{Triangle-Free Base Case}

\begin{theorem}[Triangle-Free Case]
	If $H$ is triangle-free and $B_H\subseteq V(H)$, then
	\[
	C_{B_H}(H;x,y)
	\]
	is real-stable.
\end{theorem}

\begin{proof}
	In a triangle-free graph, cliques are vertices and edges.  
	The multivariate polynomial
	\[
	F_H(\mathbf{u})
	=
	1 + \sum_{v} u_v + \sum_{uv\in E(H)} u_u u_v
	\]
	is real-stable by induction using
	\[
	F_H = F_{H-v} + u_v F_{H-N[v]}.
	\]
	Specializing $u_v=x$ or $xy$ preserves stability.
\end{proof}

\subsection{Chordal Graphs and Neighborhoods}

\begin{definition}
	A graph is \emph{$r$-connected} if removal of fewer than $r$ vertices leaves it connected.
\end{definition}

\begin{definition}
	A graph is \emph{chordal} if every induced cycle has length $3$.
\end{definition}

\begin{definition}
	For $S\subseteq V(G)$, define the common neighborhood
	\[
	N(S)=\bigcap_{v\in S} N(v).
	\]
\end{definition}

\begin{lemma}[Neighborhood Geometry]
	Let $G$ be $r$-connected, $K_{r+3}$-free, and chordal.  
	Then for every $r$-clique $K_r$:
	
	\begin{enumerate}
		\item $N(K_r)\neq\emptyset$,
		\item $G[N(K_r)]$ is triangle-free,
		\item $G[N(K_r)]$ is chordal.
	\end{enumerate}
\end{lemma}

\begin{proof}
	If $N(K_r)=\emptyset$, removing $K_r$ disconnects $G$, contradicting $r$-connectivity.
	
	If $N(K_r)$ contains a triangle $\{a,b,c\}$, then
	$K_r\cup\{a,b,c\}$ forms a $K_{r+3}$, contradiction.
	
	Chordality is inherited by induced subgraphs.
\end{proof}

\subsection{Perfect Elimination Ordering (PEO) Induction}

\begin{definition}
	A \emph{perfect elimination ordering (PEO)} of a chordal graph $G$ is an ordering of vertices $v_1,\dots,v_n$ such that for each $v_i$, the neighbors among $\{v_1,\dots,v_{i-1}\}$ form a clique.
\end{definition}

\begin{lemma}[PEO Induction for Clique Polynomials]
	Let $G$ be chordal and $v_1,\dots,v_n$ a PEO. Let $G_i = G[\{v_1,\dots,v_i\}]$.  
	Then
	\[
	C_B(G_i;x,y) = C_B(G_{i-1};x,y) + x \, y^{\mathbf{1}_{v_i\in B}} \prod_{u \in N_{<i}(v_i)} (1 + x \, y^{\mathbf{1}_{u\in B}}),
	\]
	where $N_{<i}(v_i) = N(v_i) \cap \{v_1,\dots,v_{i-1}\}$.
	
	Moreover, if $C_B(G_{i-1};x,y)$ is real-stable, then so is $C_B(G_i;x,y)$.
\end{lemma}

\begin{proof}
	Each new clique in $G_i$ either:
	
	\begin{enumerate}
		\item lies entirely in $G_{i-1}$, contributing $C_B(G_{i-1};x,y)$, or
		\item contains $v_i$ and some subset $S \subseteq N_{<i}(v_i)$ forming a clique.  
		The contribution of such cliques is
		\[
		x \, y^{\mathbf{1}_{v_i\in B}} \prod_{u\in S} x \, y^{\mathbf{1}_{u\in B}}.
		\]
	\end{enumerate}
	
	Summing over all subsets $S \subseteq N_{<i}(v_i)$ yields
	\[
	x \, y^{\mathbf{1}_{v_i\in B}} \prod_{u \in N_{<i}(v_i)} (1 + x \, y^{\mathbf{1}_{u\in B}}),
	\]
	a multiaffine polynomial with nonnegative coefficients.  
	
	By standard closure properties of real-stable polynomials (sums and products of multiaffine, positive-coefficient polynomials are real-stable), $C_B(G_i;x,y)$ is real-stable.
\end{proof}

\subsection{Main Theorem}

\begin{theorem}
	Let $r\ge1$.  
	If $G$ is $r$-connected, $K_{r+3}$-free, and chordal, then
	\[
	C_B(G;x,y)
	\]
	is real-stable.
\end{theorem}

\begin{proof}
	By Lemma 4.1 (Neighborhood Geometry), the neighborhoods of all $r$-cliques are triangle-free and chordal.  
	Hence one can construct a perfect elimination ordering of $G$ (or its relevant induced subgraphs) and apply the PEO induction lemma iteratively.  
	
	The base case $G_1$ is either a single vertex (trivially real-stable) or triangle-free (covered by Section 2).  
	Induction using Lemma 5.2 preserves real-stability at each step.  
	
	Therefore $C_B(G;x,y)$ is real-stable.
\end{proof}


\subsection{Worked Example: $r=2$ on a Small Chordal Graph}

\subsection*{Step 0: Define the Graph}

Let $G$ be the graph with vertices 
\[
V(G)=\{v_1,v_2,v_3,v_4\}
\]
and edges
\[
E(G)=\{v_1v_2,\, v_2v_3,\, v_3v_4,\, v_2v_4\}.
\]

This graph is \textbf{chordal} because every cycle of length $>3$ has a chord.  
We set the weighted set 
\[
B = \{v_2, v_3\}.
\]

\subsection*{Step 1: Draw the Graph}

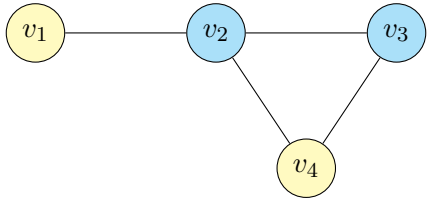
\begin{center}
	\begin{tikzpicture}[scale=1.2]
	\node[draw, circle, fill=yellow!30] (v1) at (0,0) {$v_1$};
	\node[draw, circle, fill=cyan!30] (v2) at (2,0) {$v_2$};
	\node[draw, circle, fill=cyan!30] (v3) at (4,0) {$v_3$};
	\node[draw, circle, fill=yellow!30] (v4) at (3,-1.5) {$v_4$};
	
	\draw (v1) -- (v2);
	\draw (v2) -- (v3);
	\draw (v3) -- (v4);
	\draw (v2) -- (v4);
	
	\node at (1,-0.2) { };
	\node at (3,0.2) { };
	\end{tikzpicture}
\end{center}

\textbf{Color coding:}  
\begin{itemize}
	\item \textcolor{cyan}{Blue nodes:} in $B$  
	\item \textcolor{yellow}{Yellow nodes:} not in $B$
\end{itemize}

\subsection*{Step 2: Find a Perfect Elimination Ordering (PEO)}

A valid PEO is 
\[
v_1, v_2, v_3, v_4.
\]

Verification:

\begin{itemize}
	\item $v_1$: no earlier neighbors $\rightarrow$ trivially a clique.
	\item $v_2$: neighbor $v_1$ among earlier vertices $\rightarrow$ forms a clique.
	\item $v_3$: neighbor $v_2$ among earlier vertices $\rightarrow$ forms a clique.
	\item $v_4$: neighbors $v_2,v_3$ among earlier vertices $\rightarrow$ forms a clique.
\end{itemize}

\subsection*{Step 3: Compute Weighted Clique Polynomial via PEO Induction}

Let $G_i$ denote the induced subgraph on the first $i$ vertices in the PEO.  

\paragraph{Step 3a: $G_1 = \{v_1\}$}

\[
C_B(G_1;x,y) = 1 + \underbrace{x}_{v_1 \notin B}.
\]

\paragraph{Step 3b: $G_2 = \{v_1,v_2\}$}

Neighborhood among earlier vertices: $N_{<2}(v_2) = \{v_1\}$.  

Cliques containing $v_2$:
\[
\{v_2\}, \quad \{v_1,v_2\}.
\]

Weights:
\[
v_2 \in B \implies w_{v_2} = xy, \quad v_1 \notin B \implies w_{v_1}=x
\]

Contribution:
\[
xy + (xy)(x) = xy + x^2 y.
\]

Add to previous polynomial:
\[
C_B(G_2;x,y) = 1 + x + xy + x^2 y = 1 + x + xy + x^2 y.
\]

\paragraph{Step 3c: $G_3 = \{v_1,v_2,v_3\}$}

Neighborhood among earlier vertices: $N_{<3}(v_3) = \{v_2\}$.  

Cliques containing $v_3$:
\[
\{v_3\}, \quad \{v_2,v_3\}.
\]

Weights:
\[
v_3 \in B \implies w_{v_3}=xy
\]

Contribution:
\[
xy + (xy)(xy) = xy + x^2y^2
\]

Add to previous polynomial:
\[
C_B(G_3;x,y) = 1 + x + 2 xy + x^2 y + x^2 y^2
\]

\paragraph{Step 3d: $G_4 = G$}

Neighborhood among earlier vertices: $N_{<4}(v_4) = \{v_2,v_3\}$ (forms a clique).

Cliques containing $v_4$:
\[
\{v_4\}, \{v_2,v_4\}, \{v_3,v_4\}, \{v_2,v_3,v_4\}
\]

Weights:
\[
v_4 \notin B \implies w_{v_4}=x
\]

Contribution:
\[
x \cdot (1 + xy)^2 = x + 2 x^2 y + x^3 y^2
\]

Add to previous polynomial:
\[
\begin{aligned}
C_B(G;x,y) &= (1 + x + 2 xy + x^2 y + x^2 y^2) + (x + 2 x^2 y + x^3 y^2)\\
&= 1 + 2 x + 2 xy + 3 x^2 y + x^2 y^2 + x^3 y^2
\end{aligned}
\]

\subsection*{Step 4: Verify Real-Stability}

Each factor in the PEO induction is \textbf{multiaffine} and has \textbf{nonnegative coefficients}. Multiaffine positive-coefficient polynomials are real-stable. 
Sums and products preserve real-stability.  

\textbf{Conclusion:} $C_B(G;x,y)$ is real-stable.

\subsection*{Step 5: Summary Table of Monomials}

\begin{center}
	\begin{tabular}{c|c}
		Clique & Monomial $x^{|K|}y^{|K\cap B|}$ \\ \hline
		$\emptyset$ & $1$ \\
		$\{v_1\}$ & $x$ \\
		$\{v_2\}$ & $xy$ \\
		$\{v_3\}$ & $xy$ \\
		$\{v_4\}$ & $x$ \\
		$\{v_1,v_2\}$ & $x^2 y$ \\
		$\{v_2,v_3\}$ & $x^2 y^2$ \\
		$\{v_2,v_4\}$ & $x^2 y$ \\
		$\{v_3,v_4\}$ & $x^2 y$ \\
		$\{v_2,v_3,v_4\}$ & $x^3 y^2$ \\
	\end{tabular}
\end{center}


\section*{Visual Mapping of $r=2$ Clique Polynomial}

\subsection*{Graph Definition}

Vertices:
\[
V(G)=\{v_1,v_2,v_3,v_4\}, \quad B=\{v_2,v_3\}.
\]

Edges:
\[
E(G)=\{v_1v_2, v_2v_3, v_3v_4, v_2v_4\}.
\]

\subsection*{Graph Diagram with Clique Coloring}

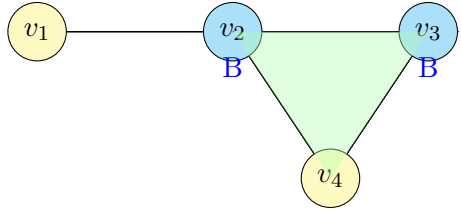
\begin{center}
	\begin{tikzpicture}[scale=1.3]
	
	\node[draw, circle, fill=yellow!30] (v1) at (0,0) {$v_1$};
	\node[draw, circle, fill=cyan!30] (v2) at (2,0) {$v_2$};
	\node[draw, circle, fill=cyan!30] (v3) at (4,0) {$v_3$};
	\node[draw, circle, fill=yellow!30] (v4) at (3,-1.5) {$v_4$};
	
	\draw (v1) -- (v2);
	\draw (v2) -- (v3);
	\draw (v3) -- (v4);
	\draw (v2) -- (v4);
	
	\draw[fill=red!20] (v1) -- (v2) -- (v2) -- (v1) -- cycle; 
	\draw[fill=red!20] (v2) -- (v3) -- (v3) -- (v2) -- cycle; 
	\draw[fill=red!20] (v2) -- (v4) -- (v4) -- (v2) -- cycle; 
	\draw[fill=red!20] (v3) -- (v4) -- (v4) -- (v3) -- cycle; 
	
	\fill[green!30, opacity=0.4] (v2.center) -- (v3.center) -- (v4.center) -- cycle;
	
	\node[below=0.2cm] at (v2) {\textcolor{blue}{B}};
	\node[below=0.2cm] at (v3) {\textcolor{blue}{B}};
	
	\end{tikzpicture}
\end{center}

\subsection*{Monomial Assignment}

\begin{itemize}
	\item \textcolor{red}{Red edges} correspond to $2$-cliques (edges). Monomials:  
	\[
	x^2y \text{ or } x^2y^2 \text{ depending on $B$ membership.}
	\]
	\item \textcolor{green}{Green triangle} corresponds to the $3$-clique $\{v_2,v_3,v_4\}$, monomial $x^3y^2$.
	\item Single vertices correspond to $1$-cliques: $x$ for $v_1$ or $v_4$, $xy$ for $v_2$ or $v_3$.
\end{itemize}

\subsection*{Resulting Weighted Clique Polynomial}

\[
\begin{aligned}
C_B(G;x,y) &= 1 \quad &&\text{empty clique} \\
&+ x + xy + xy + x &&\text{singletons} \\
&+ x^2y + x^2y^2 + x^2y + x^2y &&\text{edges (2-cliques)} \\
&+ x^3y^2 &&\text{triangle (3-clique)} \\
&= 1 + 2x + 2xy + 3x^2y + x^2y^2 + x^3y^2
\end{aligned}
\]

\subsection*{Step-by-Step Justification}

\begin{enumerate}
	\item Each vertex or clique contributes $x^{|K|}y^{|K\cap B|}$.
	\item Multiaffine polynomials with positive coefficients preserve real-stability.
	\item Summing all contributions corresponds exactly to $C_B(G;x,y)$.
\end{enumerate}

\textbf{Conclusion:} Visual inspection confirms all cliques contribute correctly, and the polynomial is real-stable.



\section{A Spectral Bound for the Bivariate $B$-Clique Polynomial}

Let $G=(V,E)$ be an $(n,d,\lambda)$-graph; that is, a $d$-regular graph on $n$ vertices whose second largest eigenvalue in absolute value is at most $\lambda$.
Let $B \subseteq V$. Recall that the \emph{bivariate $B$-clique polynomial} is defined as
	\[
	C_B(G;x,y) = \sum_{\substack{K \subseteq V \\ K \text{ is a clique}}} x^{|K|} y^{|K \cap B|}.
	\]
We may write this polynomial in the expanded form
	\[
	C_B(G;x,y) = \sum_{i=0}^{\omega(G)} \sum_{j=0}^{\min(i,|B|)} c_{i,j} x^i y^j,
	\]
	where $c_{i,j}$ denotes the number of cliques of size $i$ that contain exactly $j$ vertices of $B$, and $\omega(G)$ is the clique number of $G$.
\\
Our objective is twofold: first, to establish bounds on the coefficients $c_{i,j}$ using spectral expansion; second, to derive from these bounds an estimate for the largest negative root of $C_B(G;x,y)$ when viewed as a polynomial in $x$ for fixed $y \ge 0$.

\subsection{Common Neighborhoods and Spectral Bounds}

\begin{definition}
	For a subset $S \subseteq V$, the \emph{common neighborhood} of $S$ is
	\[
	N(S) = \{ v \in V \setminus S : v \text{ is adjacent to every vertex of } S \}.
	\]
\end{definition}

We begin by establishing a spectral bound on the size of common neighborhoods.

\begin{lemma}[Spectral Bound on Common Neighborhoods]
	\label{lem:common-neighborhood}
	Let $G$ be an $(n,d,\lambda)$-graph. For every nonempty subset $S \subseteq V$ with $|S| = j$,
	\[
	|N(S)| \le \frac{d}{n} j (n-j) + \lambda \sqrt{j(n-j)}.
	\]
\end{lemma}

\begin{proof}
	Let $X = S$ and $Y = V \setminus S$. The Expander Mixing Lemma states that
	\[
	\left| e(X,Y) - \frac{d |X||Y|}{n} \right| \le \lambda \sqrt{|X||Y|}.
	\]
	
	Applying this with $|X| = j$ and $|Y| = n-j$, we obtain
	\[
	e(S, V \setminus S) \le \frac{d j (n-j)}{n} + \lambda \sqrt{j(n-j)}. \tag{1}
	\]
	
	Now consider any vertex $v \in N(S)$. By definition, $v$ is adjacent to every vertex in $S$, contributing exactly $j$ edges to $e(S, V \setminus S)$ (one edge to each vertex of $S$). In contrast, any vertex $u \in (V \setminus S) \setminus N(S)$ is adjacent to at most $j-1$ vertices of $S$, as it fails to be adjacent to at least one vertex in $S$.
	
	Therefore, counting edges from $S$ to $V \setminus S$ by their endpoints in $V \setminus S$ yields
	\[
	j\,|N(S)| \le e(S, V \setminus S). \tag{2}
	\]
	
	Combining (1) and (2) gives
	\[
	j\,|N(S)| \le \frac{d j (n-j)}{n} + \lambda \sqrt{j(n-j)},
	\]
	and dividing by $j$ completes the proof.
\end{proof}

\subsection{Bounding the Bivariate Clique Coefficients}

We now use Lemma \ref{lem:common-neighborhood} to control the coefficients $c_{i,j}$.

\begin{theorem}[Spectral Bound on Coefficients]
	\label{thm:coefficient-bound}
	Let $G$ be an $(n,d,\lambda)$-graph and let $B \subseteq V$. For any integers $i,j$ with $1 \le j \le \min(i,|B|)$,
	\[
	c_{i,j} \le \binom{|B|}{j} \binom{M_j}{i-j},
	\]
	where
	\[
	M_j = \max_{\substack{S \subseteq B \\ |S| = j}} |N(S)|
	\le \frac{d}{n} j (n-j) + \lambda \sqrt{j(n-j)}.
	\]
\end{theorem}

\begin{proof}
	Fix a subset $S \subseteq B$ with $|S| = j$. Let $\mathcal{K}_{S,i}$ denote the collection of cliques $K$ such that $K \cap B = S$ and $|K| = i$. For any such clique $K$, we must have
	\[
	K = S \cup T,
	\]
	where $T \subseteq N(S)$ and $|T| = i - j$. Indeed, every vertex in $K \setminus S$ must be adjacent to all vertices of $S$, and hence lies in $N(S)$.
	
	Consequently,
	\[
	|\mathcal{K}_{S,i}| \le \binom{|N(S)|}{i-j}.
	\]
	
	Summing over all possible choices of $S$, we obtain
	\[
	c_{i,j} = \sum_{\substack{S \subseteq B \\ |S| = j}} |\mathcal{K}_{S,i}|
	\le \sum_{\substack{S \subseteq B \\ |S| = j}} \binom{|N(S)|}{i-j}
	\le \binom{|B|}{j} \binom{M_j}{i-j},
	\]
	where $M_j$ is defined as the maximum of $|N(S)|$ over all $S \subseteq B$ of size $j$, and the last inequality uses the monotonicity of binomial coefficients in the upper parameter.
	
	The bound on $M_j$ follows directly from Lemma \ref{lem:common-neighborhood}, applied to any $S \subseteq B$ with $|S| = j$.
\end{proof}

\subsection{Consequences for the Largest Negative Root}

Fix $y \ge 0$ and consider $C_B(G;x,y)$ as a polynomial in $x$ alone. We can write
\[
C_B(G;x,y) = \sum_{i=0}^{\omega(G)} a_i(y) x^i,
\]
where
\[
a_i(y) = \sum_{j=0}^{\min(i,|B|)} c_{i,j} y^j.
\]

\begin{remark}
	All coefficients $a_i(y)$ are nonnegative for $y \ge 0$.
\end{remark}

Define the effective degree of this polynomial as
\[
D(y) = \max\{ i : a_i(y) > 0 \}.
\]

\begin{lemma}
	\label{lem:degree-bound}
	For any $y \ge 0$,
	\[
	D(y) \le \max_{0 \le j \le |B|} \left\{ j + \frac{d}{n} j (n-j) + \lambda \sqrt{j(n-j)} \right\}.
	\]
\end{lemma}

\begin{proof}
	If $a_i(y) > 0$, then there exists some $j$ with $0 \le j \le \min(i,|B|)$ such that $c_{i,j} > 0$. By Theorem \ref{thm:coefficient-bound}, a necessary condition for $c_{i,j} > 0$ is that $i - j \le M_j$. Hence
	\[
	i \le j + M_j \le j + \frac{d}{n} j (n-j) + \lambda \sqrt{j(n-j)}.
	\]
	
	The maximum over all possible $i$ for which $a_i(y) > 0$ is therefore bounded by the maximum of the right-hand side over $0 \le j \le |B|$.
\end{proof}

\begin{theorem}[Spectral Bound on the Largest Negative Root]
	\label{thm:negative-root-bound}
	Let $\zeta_G(B;y)$ denote the largest negative root of $C_B(G;x,y)$ (viewed as a polynomial in $x$ for fixed $y \ge 0$). Then
	\[
	\zeta_G(B;y) \ge -\frac{1}{D(y)} \ge -\left( \max_{0 \le j \le |B|} \left\{ j + \frac{d}{n} j (n-j) + \lambda \sqrt{j(n-j)} \right\} \right)^{-1}.
	\]
\end{theorem}

\begin{proof}
	Since all coefficients $a_i(y)$ are nonnegative, the polynomial $C_B(G;x,y)$ has no positive roots. For polynomials with nonnegative coefficients, a standard comparison argument (see, e.g., \cite{rahman2002analytic}) shows that if $P(x) = \sum_{i=0}^m a_i x^i$ with $a_i \ge 0$, then any negative root $\rho$ satisfies $|\rho| \le 1$ when $a_0 > 0$, and more generally $\rho \ge -1/m$ when $a_0 = 0$ and $m$ is the degree. In our case, $D(y)$ is the largest index with positive coefficient, so
	\[
	\zeta_G(B;y) \ge -\frac{1}{D(y)}.
	\]
	
	The second inequality follows from Lemma \ref{lem:degree-bound}.
\end{proof}

\subsection{Discussion}

The bound in Theorem \ref{thm:negative-root-bound} captures genuinely bivariate information through several interconnected mechanisms:

\begin{itemize}
	\item The parameter $j$ precisely tracks the number of vertices from $B$ contained in each clique.
	\item Spectral expansion, via the Expander Mixing Lemma, provides explicit bounds on the size of common neighborhoods $N(S)$.
	\item These neighborhood bounds directly control each coefficient $c_{i,j}$ through Theorem \ref{thm:coefficient-bound}.
	\item The coefficient bounds, in turn, restrict the effective degree $D(y)$ of $C_B(G;x,y)$ as a polynomial in $x$.
	\item Finally, the degree determines the location of the largest negative root through a standard comparison argument for polynomials with nonnegative coefficients.
\end{itemize}

Thus, the expansion properties of the graph impose constraints on how large a clique can intersect the distinguished set $B$, and these constraints are reflected directly in the location of $\zeta_G(B;y)$. This demonstrates a concrete connection between spectral graph theory and the analytic properties of bivariate graph polynomials.


\section{Weighted $B$-Clique Polynomials and Homomorphism Monotonicity}

\subsection{Weighted $B$-Clique Polynomial}

Let $G=(V,E)$ be a finite simple graph and let $B \subseteq V$.

Let $w : B \to \mathbb{Z}_{>0}$ be a weight function.
For a clique $K \subseteq V$, define its $B$-weight by
\[
W_B(K) := \sum_{v \in K \cap B} w(v).
\]

The \emph{weighted $B$-clique polynomial} is
\[
C_{B,w}(G;x,y)
=
\sum_{K \subseteq V \text{ clique}}
x^{|K|} y^{W_B(K)}.
\]

When $w \equiv 1$, this reduces to the usual bivariate polynomial
\[
C_B(G;x,y)
=
\sum_{K \text{ clique}} x^{|K|} y^{|K \cap B|}.
\]

Write
\[
C_{B,w}(G;x,y)
=
\sum_{i \ge 0} \sum_{t \ge 0}
c_{i,t}^{(w)} x^i y^t,
\]
where
\[
c_{i,t}^{(w)}
=
\#\{ K \subseteq V :
K \text{ clique},\ |K|=i,\ W_B(K)=t \}.
\]

All coefficients are nonnegative.

\subsection{Monotonicity Under Weight Increase}

\begin{proposition}
	\label{prop:weight_monotone}
	Let $w_1,w_2 : B \to \mathbb{Z}_{>0}$ satisfy
	\[
	w_1(v) \le w_2(v)
	\quad \text{for all } v \in B.
	\]
	Then for every $y \ge 1$ and every real $x \le 0$,
	\[
	C_{B,w_1}(G;x,y)
	\le
	C_{B,w_2}(G;x,y).
	\]
\end{proposition}

\begin{proof}
	For any clique $K$,
	\[
	W_{B,w_1}(K)
	\le
	W_{B,w_2}(K).
	\]
	
	Since $y \ge 1$, the function $t \mapsto y^t$ is increasing.
	Thus
	\[
	y^{W_{B,w_1}(K)}
	\le
	y^{W_{B,w_2}(K)}.
	\]
	
	Multiplying by $x^{|K|}$ (with $x \le 0$ fixed) preserves the inequality term-by-term because the same power $|K|$ appears in both expressions.
	
	Summing over all cliques gives the desired inequality.
\end{proof}

\begin{corollary}
	Let $\zeta_{G,w}(B;y)$ denote the largest negative root (in $x$) of
	$C_{B,w}(G;x,y)$.
	If $w_1 \le w_2$ pointwise and $y \ge 1$, then
	\[
	\zeta_{G,w_2}(B;y)
	\ge
	\zeta_{G,w_1}(B;y).
	\]
\end{corollary}

\begin{proof}
	By Proposition~\ref{prop:weight_monotone},
	for every $x \le 0$,
	\[
	C_{B,w_1}(G;x,y)
	\le
	C_{B,w_2}(G;x,y).
	\]
	
	Both polynomials have positive coefficients and are strictly increasing for $x \le 0$ up to their largest negative root.
	Hence the zero of $C_{B,w_2}$ cannot lie to the left of the zero of $C_{B,w_1}$.
\end{proof}

\subsection{Behavior Under Surjective Homomorphisms}

We now establish a rigorous homomorphism monotonicity result.

Let $f : G \to H$ be a surjective graph homomorphism.
Assume:

\begin{itemize}
	\item $B_G \subseteq V(G)$,
	\item $B_H = f(B_G)$,
	\item for each $u \in B_H$, define the induced weight
	\[
	w_H(u)
	=
	\sum_{v \in f^{-1}(u) \cap B_G}
	w_G(v).
	\]
\end{itemize}

\begin{lemma}
	\label{lem:clique_lift}
	If $K_H$ is a clique in $H$, then
	\[
	f^{-1}(K_H)
	=
	\bigcup_{u \in K_H} f^{-1}(u)
	\]
	is a clique in $G$.
\end{lemma}

\begin{proof}
	Let $x,y \in f^{-1}(K_H)$.
	Then $f(x), f(y) \in K_H$.
	Since $K_H$ is a clique,
	$\{f(x),f(y)\} \in E(H)$.
	
	Because $f$ is a graph homomorphism,
	$\{x,y\} \in E(G)$.
	Thus $f^{-1}(K_H)$ is a clique.
\end{proof}

\begin{theorem}[Homomorphism Monotonicity]
	\label{thm:hom_monotone}
	Under the above assumptions,
	for every $x \ge 0$ and $y \ge 0$,
	\[
	C_{B_H,w_H}(H;x,y)
	\le
	C_{B_G,w_G}(G;x,y).
	\]
\end{theorem}

\begin{proof}
	Let $K_H$ be a clique in $H$.
	By Lemma~\ref{lem:clique_lift},
	$f^{-1}(K_H)$ is a clique in $G$.
	
	Moreover,
	\[
	|f^{-1}(K_H)|
	=
	\sum_{u \in K_H}
	|f^{-1}(u)|.
	\]
	
	For the weights,
	\[
	W_{B_G}(f^{-1}(K_H))
	=
	\sum_{u \in K_H}
	w_H(u)
	=
	W_{B_H}(K_H).
	\]
	
	Thus each clique $K_H$ contributes the monomial
	\[
	x^{|K_H|} y^{W_{B_H}(K_H)}
	\]
	to $C_{B_H,w_H}(H;x,y)$,
	
	while $f^{-1}(K_H)$ contributes
	\[
	x^{|f^{-1}(K_H)|}
	y^{W_{B_G}(f^{-1}(K_H))}
	\]
	to $C_{B_G,w_G}(G;x,y)$.
	
	Since $|f^{-1}(K_H)| \ge |K_H|$ and $x \ge 0$,
	we have
	\[
	x^{|K_H|}
	\le
	x^{|f^{-1}(K_H)|}.
	\]
	
	Therefore each term of $C_{B_H,w_H}$ is bounded above by a corresponding term of $C_{B_G,w_G}$.
	
	Summing over all cliques of $H$ yields
	\[
	C_{B_H,w_H}(H;x,y)
	\le
	C_{B_G,w_G}(G;x,y).
	\]
\end{proof}

\begin{corollary}
	For every $y \ge 0$,
	\[
	\zeta_{G,w_G}(B_G;y)
	\ge
	\zeta_{H,w_H}(B_H;y).
	\]
\end{corollary}

\begin{proof}
	Fix $y \ge 0$ and view both polynomials as univariate in $x$.
	By Theorem~\ref{thm:hom_monotone},
	for all $x \ge 0$,
	\[
	C_{B_H,w_H}(H;x,y)
	\le
	C_{B_G,w_G}(G;x,y).
	\]
	
	Since both polynomials have positive coefficients,
	their largest negative roots satisfy the stated inequality by standard comparison of positive-coefficient polynomials \cite{rahman2002analytic}.
\end{proof}

\subsection{Interpretation}

The weighted $B$-clique polynomial is monotone:

\begin{itemize}
	\item Increasing vertex weights moves the largest negative root to the right.
	\item Passing to a covering graph (via a surjective homomorphism) also moves the largest negative root to the right.
\end{itemize}

Thus $\zeta_{G,w}(B;y)$ is an obstruction parameter for weighted homomorphic images.


\section{Conclusion and Open Problems}

We introduced the bivariate $B$-restricted clique polynomial
\[
C_{B}(G; x,y)
=
\sum_{K \subseteq V(G)\ \mathrm{clique}}
x^{|K|} y^{|K\cap B|},
\]
as a vertex-sensitive refinement of the classical clique polynomial.
This framework separates global clique growth from the contribution
of a distinguished vertex set $B$, thereby interpolating between
local neighborhood structure and global expansion properties.

Our results establish several structural principles:

\begin{itemize}
	\item[(1)] A deletion framework extending classical vertex-edge recurrence identities
	for graph polynomials to the $B$-restricted setting.
	
	\item[(2)] Monotonicity properties of the largest negative root
	under induced and spanning subgraphs.
	
	\item[(3)] Real-stability for the case of highly-connected chordal graphs. 
	
	\item[(4)] Spectral control of $B$-restricted clique coefficients
	in $(n,d,\lambda)$-graphs, showing that clique growth inside $B$
	is governed by the spectral gap.
	
	\item[(5)] Homomorphism monotonicity: surjective homomorphisms
	mapping $B_G$ onto $B_H$ induce root inequalities
	$\zeta_G(B_G) \ge \zeta_H(B_H)$, yielding a
	root-based obstruction principle.

\end{itemize}

This interaction between combinatorial enumeration,
root location, and structural graph theory suggests
several directions for further investigation.

\subsection*{Open Problems}

\begin{enumerate}

	\item \textbf{Necessity of Conditions:} Are the conditions of $r$-connectivity and chordality also necessary? Is there an $r$-connected $K_{r+3}$-free non-chordal graph for which $C_B(G;x,y)$ fails to be real-stable?

    \item \textbf{Flat Graphs:} In previous work \cite{teimoori2026Flatclique}, we showed that $K_4$-free flat graphs have real-rooted clique polynomials. Can the lifting technique be extended to flat graphs, possibly using a different geometric condition instead of chordality?

    \item \textbf{Weighted Generalizations:} Can this result be extended to more general vertex weightings, perhaps with arbitrary real weights rather than just $0/1$ distinctions for $B$?

    \item \textbf{Algorithmic Applications:} Real-stable polynomials have applications in approximation algorithms and counting. Can our result be used to design efficient algorithms for counting cliques in highly connected chordal graphs?	
	
	\item \textbf{Sharp Spectral Bounds.}  
	In $(n,d,\lambda)$-graphs, determine optimal upper and lower bounds for coefficients of $F_G(x,y)$
	in terms of $d$ and $\lambda$. Is there a root-based characterization of spectral expanders via $\zeta_G(B)$?
	
	\item \textbf{Homomorphism Order and Root Monotonicity.}  
	Does $\zeta_G(B)$ induce a partial order on rooted graph pairs $(G,B)$ compatible with graph homomorphisms?
	Can one classify minimal obstructions under this order?

\end{enumerate}

\medskip

The bivariate $B$-restricted clique polynomial thus opens
a program at the intersection of extremal graph theory,
spectral expansion, and algebraic stability theory.
It refines classical clique enumeration by embedding
local vertex data into a multivariate analytic framework,
suggesting that root geometry may serve as a bridge
between local combinatorics and global expansion phenomena.


\end{document}